\documentclass[oneside,english]{amsart}
\usepackage[T1]{fontenc}
\usepackage[latin9]{inputenc}
\usepackage{amsthm}
\usepackage{amssymb}

\makeatletter
\numberwithin{equation}{section}
\numberwithin{figure}{section}
\theoremstyle{plain}
\newtheorem{thm}{\protect\theoremname}[section]
  \theoremstyle{plain}
  \newtheorem{conjecture}[thm]{\protect\conjecturename}
  \theoremstyle{definition}
  \newtheorem{problem}[thm]{\protect\problemname}
  \theoremstyle{definition}
  \newtheorem{defn}[thm]{\protect\definitionname}
  \theoremstyle{plain}
  \newtheorem{lem}[thm]{\protect\lemmaname}

\AtBeginDocument{

}

\makeatother

\usepackage{babel}
  \providecommand{\conjecturename}{Conjecture}
  \providecommand{\definitionname}{Definition}
  \providecommand{\lemmaname}{Lemma}
  \providecommand{\problemname}{Problem}
\providecommand{\theoremname}{Theorem}

\begin{document}
\vspace{0.5in}

\title{Combined Properties of Finite Sums \& Finite Products near Zero }

\author{Tanushree Biswas}

\address{Department of Mathematics, University of Kalyani, Kalyani-741235,
West Bengal, India}

\email{tanushreebiswas87@gmail.com}

\begin{abstract}
\textsl{\emph{It was proved that whenever $\mathbb{N}$ is partitioned
into finitely many cells, one cell must contain arbitrary length geo-arithmetic
progressions. It was also proved that arithmetic and geometric progressions
can be nicely intertwined in one cell of partition, whenever $\mathbb{N}$
is partitioned into finitely many cells. In this article we shall
prove that similar types of results also hold near zero in some suitable
dense subsemigroup $S$ of $\left(\left(0,\,\infty\right),\,+\right)$,
using the }}Stone-\v{C}ech compactification $\beta S$.
\end{abstract}

\maketitle

\section{Introduction}

One of the famous Ramsey theoretic results is van der Waerden's Theorem
\cite{waerden1927van}, which says that whenever the set $\mathbb{N}$
of natural numbers is divided into finitely many classes, one of these
classes contains arbitrarily long arithmetic progressions. The analogous
statement about geometric progressions is easily seen to be equivalent
via the homomorphisms $p\colon\,\left(\mathbb{N},\,+\right)\to\left(\mathbb{N},\,\cdotp\right)$
such that $p\left(x\right)=2^{x}$, and $q\colon\,\left(\mathbb{N}\setminus\left\{ 1\right\} ,\,\cdotp\right)\to\left(\mathbb{N},\,+\right)$,
where $q\left(x\right)$ is the length of the prime factorization
of $x$.

It has been shown in \cite[Theorem 3.11]{beiglbock2008some} that
any set which is multiplicatively large, that is a piecewise syndetic
IP set in $\left(\mathbb{N},\,\cdotp\right)$ must contain substantial
combined additive and multiplicative structure; in particular it must
contain arbitrarily large geo-arithmetic progressions, that is, sets
of the form
\[
\left\{ r^{j}\left(a+id\right)\colon\,i,j\in\left\{ 1,2,\ldots,k\right\} \right\} .
\]
 A well known extension of van der Waerden's Theorem allows one to
get the additive increment of the arithmetic progression in the same
cell as the arithmetic progression. Similarly for any finite partition
of $\mathbb{N}$ there exist some cell $A$ and $b,r\in\mathbb{N}$
such that $\left\{ r,b,br,\ldots,br^{k}\right\} \subseteq A$. It
is proved in \cite[Theorem 1.5]{beiglbock2008some} that these two
facts can be intertwined:
\begin{thm}
Let $r,k\in\mathbb{N}$ and $\mathbb{N}=\bigcup_{i=1}^{r}A_{i}$.
Then there exist $s\in\left\{ 1,2,\dots,r\right\} $and $a,b,d\in A_{s}$,
such that\\
\begin{align*}
\left\{ b\left(a+id\right)^{j}\colon\,i,j\in\left\{ 0,1,\ldots,k\right\} \right\} \bigcup\left\{ bd^{j}\colon\,j\in\left\{ 0,1,\ldots,k\right\} \right\} \\
\bigcup\left\{ a+id\colon\,i\in\left\{ 0,1,\ldots,k\right\} \right\}  & \subseteq A_{s}
\end{align*}
\end{thm}
We know that if $A\subseteq\mathbb{N}$ belongs to every idempotent
in $\beta\mathbb{N}$, then it is called an $IP^{*}$ set. Given a
sequence $\langle x_{n}\rangle_{n=1}^{\infty}$ in $\mathbb{N}$,
we let $FP(\langle x_{n}\rangle_{n=1}^{\infty})$ be the product analogue
of Finite Sum. Given a sequence $\langle x_{n}\rangle_{n=1}^{\infty}$
in $\mathbb{N}$, we say that $\langle y_{n}\rangle_{n=1}^{\infty}$
is a \textit{sum subsystem} of $\langle x_{n}\rangle_{n=1}^{\infty}$
provided there is a sequence $\langle H_{n}\rangle_{n=1}^{\infty}$
of nonempty finite subsets of $\mathbb{N}$ such that $\max H_{n}<\min H_{n+1}$,
and $y_{n}=\sum_{t\in H_{n}}\,x_{t}$ for each $n\in\mathbb{N}$.
\begin{thm}
\label{ip*-1} Let $\langle x_{n}\rangle_{n=1}^{\infty}$ be a sequence
in $\mathbb{N}$ and $A$ be an IP{*} set in $(\mathbb{N},+)$. Then
there exists a sum subsystem $\langle y_{n}\rangle_{n=1}^{\infty}$
of $\langle x_{n}\rangle_{n=1}^{\infty}$ such that 
\[
FS(\langle y_{n}\rangle_{n=1}^{\infty})\cup FP(\langle y_{n}\rangle_{n=1}^{\infty})\subseteq A.
\]
\end{thm}
\begin{proof}
\cite[Theorem 2.6]{bergelson1994ip} or see \cite[Corollary 16.21]{hindman1998algebra}. 
\end{proof}
The algebraic structure of the smallest ideal of $\beta S$ has played
a significant role in Ramsey Theory. It is known that any central
subset of $(\mathbb{N},+)$ is guaranteed to have substantial additive
structure. But Theorem 16.27 of \cite{hindman1998algebra} shows that
central sets in $(\mathbb{N},+)$ need not have any multiplicative
structure at all. On the other hand, in \cite{bergelson1994ip} we
see that sets which belong to every minimal idempotent of N, called
central{*} sets, must have significant multiplicative structure.

In case of central{*} sets a similar result has been proved in \cite{de2007combined}
for a restricted class of sequences called minimal sequences, where
a sequence $\langle x_{n}\rangle_{n=1}^{\infty}$ in $\mathbb{N}$
is said to be a minimal sequence if 
\[
\bigcap_{m=1}^{\infty}\overline{FS(\langle x_{n}\rangle_{n=m}^{\infty})}\cap K(\beta\mathbb{N})\neq\emptyset.
\]

\begin{thm}
\label{central*} Let $\langle y_{n}\rangle_{n=1}^{\infty}$ be a
minimal sequence and $A$ be a $central^{*}$ set in $(\mathbb{N},+)$.
Then there exists a sum subsystem $\langle x_{n}\rangle_{n=1}^{\infty}$
of $\langle y_{n}\rangle_{n=1}^{\infty}$ such that 
\[
FS(\langle x_{n}\rangle_{n=1}^{\infty})\cup FP(\langle x_{n}\rangle_{n=1}^{\infty})\subseteq A.
\]
\end{thm}
\begin{proof}
\cite[Theorem 2.4]{bergelson1994ip}.
\end{proof}
A similar result in this direction in the case of dyadic rational
numbers has been proved by Bergelson, Hindman and Leader.
\begin{thm}
\label{notramseyd} There exists a finite partition $\mathbb{D}\setminus\{0\}=\bigcup_{i=1}^{r}A_{i}$
such that there do not exist $i\in\{1,2,\ldots,r\}$ and a sequence
$\langle x_{n}\rangle_{n=1}^{\infty}$ with 
\[
FS(\langle x_{n}\rangle_{n=1}^{\infty})\cup FP(\langle x_{n}\rangle_{n=1}^{\infty})\subseteq A_{i}.
\]
\end{thm}
\begin{proof}
\cite[Theorem 5.9]{bergelson1999additive}. 
\end{proof}
In \cite{bergelson1999additive}, the authors also presented the following
conjecture and question.
\begin{conjecture}
There exists a finite partition $\mathbb{Q}\setminus\{0\}=\bigcup_{i=1}^{r}A_{i}$
such that there do not exists $i\in\{1,2,\ldots,r\}$ and a sequence
$\langle x_{n}\rangle_{n=1}^{\infty}$ with 
\[
FS(\langle x_{n}\rangle_{n=1}^{\infty})\cup FP(\langle x_{n}\rangle_{n=1}^{\infty})\subseteq A_{i}.
\]
\end{conjecture}
\begin{problem}
Does there exists a finite partition $\mathbb{R}\setminus\{0\}=\bigcup_{i=1}^{r}A_{i}$
such that there do not exists $i\in\{1,2,\ldots,r\}$ and a sequence
$\langle x_{n}\rangle_{n=1}^{\infty}$ with 
\[
FS(\langle x_{n}\rangle_{n=1}^{\infty})\cup FP(\langle x_{n}\rangle_{n=1}^{\infty})\subseteq A_{i}?
\]
\end{problem}
In the section 2, we shall first work on some combined algebraic properties
near $0$ in the ring of quaternions, denoted by $\mathbb{H}$. The
ring being non abelian, is a division ring having an idempotent $0$.
In section 3, for any suitable dense subsemigroup $S$ of $\left(\left(0,\infty\right),+\right)$,
our aim is to establish partitition regularity among two matrices
using using additive and multiplicative structure of $\beta S$, Stone-\v{C}ech
compactification of $S$.

\section{Combined Algebraic and Multiplicative Properties near idempotent
in relation with Quaternion Rings}

In the following discussion our aim is to extend Theorem \ref{ip*-1}
and Theorem \ref{central*} for dense subsemigroups $(\mathbb{H},+)$
in the appropriate context.
\begin{defn}
If $S$ is a dense subsemigroup of $(\mathbb{H},+)$, we define $0^{+}(S)=\{p\in\beta S_{d}:(\forall r>0)(B_{d}(r)\in p)\}$. 
\end{defn}
It is proved in \cite{hindman1999semigroup}, that $0^{+}(S)$ is
a compact right topological subsemigroup of $(\beta S_{d},+)$ which
is disjoint from $K(\beta S_{d})$ and hence gives some new information
which are not available from $K(\beta S_{d})$. Being compact right
topological semigroup $0^{+}(S)$ contains minimal idempotents of
$0^{+}(S)$. A subset $A$ of $S$ is said to be IP{*}-set near 0
if it belongs to every idempotent of $0^{+}(S)$ and a subset $C$
of $S$ is said to be central$^{*}$ set near $0$ if it belongs to
every minimal idempotent of $0^{+}(S)$. In \cite{de2009image} the
authors applied the algebraic structure of $0^{+}(S)$ on their investigation
of image partition regularity near $0$ of finite and infinite matrices.
Article \cite{de2011image} used algebraic structure of $0^{+}(\mathbb{R})$
to investigate image partition regularity of matrices with real entries
from $\mathbb{R}$.
\begin{defn}
Let $S$ be a dense subsemigroup of $(\mathbb{H},+)$. A subset $A$
of $S$ is said to be an IP set near $0$ if there exists a sequence
$\langle x_{n}\rangle_{n=1}^{\infty}$ with $\sum_{n=1}^{\infty}x_{n}$
converges such that $FS(\langle x_{n}\rangle_{n=1}^{\infty})\subseteq A$.
We call a subset $D$ of $S$ is an IP$^{*}$ set near $0$ if for
every subset $C$ of $S$ which is IP set near $0$, $C\cap D$ is
IP set near $0$. 
\end{defn}
From \cite[Theorem 3.2]{akbari2012semigroup}, it follows that for
a dense subsemigroup $S$ of $(\mathbb{H},+)$ a subset $A$ of $S$
is an \emph{IP set near $0$} if only if there exists some idempotent
$p\in0^{+}(S)$ with $A\in p$. Further it can be easily observed
that a subset $D$ of $S$ is an IP$^{*}$ set near $0$ if and only
if it belongs to every idempotent of $0^{+}(S)$.

Given $c\in\mathbb{H}\setminus\{0\}$ and $p\in\beta\mathbb{H}_{d}$,
the product $c\cdot p$ and $p\cdot c$ are defined in $(\beta\mathbb{H}_{d},\cdot)$.
One has $A\subseteq\mathbb{H}$ is a member of $c\cdot p$  if and
only if $c^{-1}A=\{x\in\mathbb{H}:c\cdot x\in A\}$ is a member of
$p$ and similarly for $c\cdot p$.
\begin{lem}
\label{ip*0} Let $S$ be a dense subsemigroup of $(\mathbb{H},+)$
such that $S\cap\mathbb{H}$ is a subsemigroup of $(\mathbb{H}\setminus\{0\},\cdot)$.
If $A$ is an IP set near $0$ in $S$ then $sA$ is also an IP set
near $0$ for every $s\in S\cap B_{d}(1)\setminus\{0\}$. Further
if $A$ is a an IP$^{*}$ set near $0$ in $(S,+)$ then both $s^{-1}A$
and $As^{-1}$ are IP$^{*}$ set near $0$ for every $s\in S\cap B_{d}(1)\setminus\{0\}$.
\end{lem}
\begin{proof}
Since $A$ is an IP set near $0$ then by \cite[Theorem 3.1]{hindman1999semigroup}
there exists a sequence $\langle x_{n}\rangle_{n=1}^{\infty}$ in
$S$ with the property that $\sum_{n=1}^{\infty}x_{n}$ converges
and FS$(\langle x_{n}\rangle_{n=1}^{\infty})\subseteq A$. This implies
that $\sum_{n=1}^{\infty}(s\cdot x_{n})$ is also convergent and FS$(\langle sx_{n}\rangle_{n=1}^{\infty})\subseteq sA$.
This proves that $sA$ is also an IP set near $0$. Similarly we can
prove that $As^{-1}$ is also IP set near $0$ for every $s\in S\cap B_{d}(1)\setminus\{0\}$.
For the second let $A$ be a an IP$^{*}$ set near $0$ and $s\in S\cap B_{d}(1)\setminus\{0\}$.
To prove that $s^{-1}A$ is a an IP$^{*}$ set near $0$ it is sufficient
to show that if $B$ is any IP set near $0$ then $B\cap s^{-1}A\neq\emptyset$.
Since $B$ is an IP set near $0$, $sB$ is also an IP set near $0$
by the first part of the proof, so that $A\cap sB\neq\emptyset$.
Choose $t\in sB\cap A$ and $k\in B$ such that $t=sk$. Therefore
$k\in s^{-1}A$ so that $B\cap s^{-1}A\neq\emptyset$. 
\end{proof}
Given $A\subseteq S$ and $s\in S$, $s^{-1}A=\{t\in S$ : $st\in A\}$
and $-s+A=\{t\in S$ : $s+t\in A\}$. In case of product we have to
keep in mind the order of elements as the product is noncommutative
here.
\begin{defn}
Let $\langle x_{n}\rangle_{n=1}^{\infty}$ be a sequence in the ring
$(\mathbb{H},+,\cdot)$, and let $k\in\mathbb{N}$. Then FP$(\langle x_{n}\rangle_{n=1}^{k})$
is the set of all products of terms of $\langle x_{n}\rangle_{n=1}^{k}$
in any order with no repetitions. Similarly FP$(\langle x_{n}\rangle_{n=1}^{\infty})$
is the set of all products of terms of $\langle x_{n}\rangle_{n=1}^{\infty}$
in any order with no repetitions.
\end{defn}
\begin{thm}
\label{ip*comb} Let $S$ be a dense subsemigroup of $(\mathbb{H},+)$,
such that $S\cap B_{d}(1)\setminus\{0\}$ is a subsemigroup of $(B_{d}(1)\setminus\{0\},\cdot)$.
Also let $\langle x_{n}\rangle_{n=1}^{\infty}$ be a sequence in $S$
such that $\sum_{n=1}^{\infty}x_{n}$ converges to $0$ and $A$ be
a IP$^{*}$ set near $0$ in $S$. Then there exists a sum subsystem
$\langle y_{n}\rangle_{n=1}^{\infty}$ of $\langle x_{n}\rangle_{n=1}^{\infty}$
such that 
\[
FS(\langle y_{n}\rangle_{n=1}^{\infty})\cup FP(\langle y_{n}\rangle_{n=1}^{\infty})\subseteq A.
\]
\end{thm}
\begin{proof}
Since $\sum_{n=1}^{\infty}x_{n}$ converges to 0, from \cite[Theorem 3.1]{hindman1999semigroup}
it follows that we can find some idempotent $p\in0^{+}(S)$ for which
$FS(\langle x_{n}\rangle_{n=1}^{\infty})\in p$. In fact $T=\bigcap_{m=1}^{\infty}\mbox{cl}_{\beta S_{d}}FS(\langle y_{n}\rangle_{n=m}^{\infty})\subseteq0^{+}(S)$
and $p\in T$. Again, since $A$ is an IP{*} set near $0$ in $S$,
by the above Lemma \ref{ip*0} for every $s\in S\cap B_{d}(1)\setminus\{0\}$,
both $s^{-1}A,As^{-1}\in p$. Let $A^{\star}=\{s\in A$ : $-s+A\in p\}$.
Then by ~\cite[Lemma 4.14]{hindman1998algebra} $A^{\star}\in p$.
We can choose $y_{1}\in A^{\star}\cap FS(\langle x_{n}\rangle_{n=1}^{\infty})$.
Inductively let $m\in\mathbb{N}$ and $\langle y_{i}\rangle_{i=1}^{m}$,
$\langle H_{i}\rangle_{i=1}^{m}$ in $\mathcal{P}_{f}(\mathbb{N})$
be chosen with the following properties:

\begin{enumerate}
\item $i\in\{1,2,\ldots,m-1\}$ $\max H_{i}<\min H_{i+1}$; 
\item If $y_{i}=\sum_{t\in H_{i}}x_{t}$ then $\sum_{t\in H_{m}}x_{t}\in A^{\star}$
and $\mbox{FP}(\langle y_{i}\rangle_{i=1}^{m})\subseteq A$. 
\end{enumerate}
We observe that $\{\sum_{t\in H}x_{t}$ : $H\in\mathcal{P}_{f}(\mathbb{N}),\min H>\max H_{m}\}\in p$.
Let $B=\{\sum_{t\in H}x_{t}$ : $H\in\mathcal{P}_{f}(\mathbb{N}),\min H>\max H_{m}\}$,
let $E_{1}=FS(\langle y_{i}\rangle_{i=1}^{m})$ and $E_{2}=AP(\langle y_{i}\rangle_{i=1}^{m})$.
Now consider 
\[
D=B\cap A^{\star}\cap\bigcap_{s\in E_{1}}(-s+A^{\star})\cap\bigcap_{s\in E_{2}}(s^{-1}A^{\star})\cap\bigcap_{s\in E_{2}}(A^{\star}s^{-1}).
\]
Then $D\in p$. Now choose $y_{m+1}\in D$ and $H_{m+1}\in\mathcal{P}_{f}(\mathbb{N})$
such that $\min H_{m+1}>\max H_{m}$. Putting $y_{m+1}=\sum_{t\in H_{m+1}}x_{t}$
shows that the induction can be continued and this proves the theorem. 
\end{proof}

\section{An Application of additive and multiplicative structure of $\beta S$}

We shall like to produce an alternative proof of the above Theorem
\ref{extensionipr0} using additive and multiplicative structure of
$\beta S$. We need the following notion.
\begin{thm}
\label{extensionipr0} Let $u,\,v\in\mathbb{N}$. Let $M$ be a finite
image partition regular matrix over $\mathbb{N}$ of order $u\times v$,
and let $N$ be an infinite image partition regular near $0$ matrix
over a dense subsemigroup $S$ of $((0,\infty),+)$. Then 
\[
\left(\begin{array}{cc}
M & O\\
O & N
\end{array}\right)\,
\]
 is image partition regular near $0$ over $S$.
\end{thm}
\begin{defn}
Let $S$ be a subsemigroup of $((0,\infty),+)$ and let $A$ be a
finite or infinite matrix with entries from $\mathbb{Q}$. Then $I(A)=\{p\in0^{+}$
: for every $P\in p$, there exists $\vec{x}$ with entries from $S$
such that all entries of $A\vec{x}$ are in $P\}$.
\end{defn}
The following lemma can be easily proved as \cite[Lemma 2.5]{hindman2003infinite}.
\begin{lem}
Let $A$ be a matrix, finite or infinite with entries from $\mathbb{Q}$.

(a) The set $I(A)$ is compact and $I(A)\neq\emptyset$ if and only
if $A$ is image partition regular near $0$.

(b) If $A$ is finite image partition regular matrix, then$I(A)$
is a sub-semigroup of $(0^{+},+)$.
\end{lem}
Next, we shall investigate the multiplicative structure of $I(A)$.
In the following Lemma \ref{lem:I(A)}, we shall see that if $A$
is a image partition regular near $0$ then $I(A)$ is a left ideal
of $(0^{+},\cdot)$. It is also a two-sided ideal of $(0^{+},\cdot)$,
provided $A$ is a finite image partition regular near $0$.
\begin{lem}
\label{lem:I(A)} Let $A$ be a matrix, finite or infinite with entries
from $\mathbb{Q}$.

(a) If $A$ is an image partition regular near $0$, then $I(A)$
is a left ideal of $(0^{+},\cdot)$.

(b) If $A$ is a finite image partition regular near $0$, then $I(A)$
is a two-sided ideal of $(0^{+},\cdot)$.
\end{lem}
\begin{proof}
(a). Let  $A$ be a $u\times v$ image partition regular matrix, where
$u,v\in\mathbb{N}\cup\{\omega\}$. Let $p\in0^{+}$ and $q\in I(A)$.
Also let $U\in p\cdot q$. Then\\
 $\{x\in S:\,x^{-1}U\in q\}\in p$. Choose $z\in\{x\in S:\,x^{-1}U\in q\}$.
Then $z^{-1}U\in q$. So there exists $\vec{x}$ with entries from
$S$ such that $y_{j}\in z^{-1}U$ for $0\leq j<u$ where $\vec{y}=A\vec{x}$,

$\vec{y}=\left(\begin{matrix}y_{0}\\
y_{1}\\
y_{2}\\
\vdots
\end{matrix}\right)$and $\vec{x}=\left(\begin{matrix}x_{0}\\
x_{1}\\
x_{2}\\
\vdots
\end{matrix}\right)$; $\vec{y}$ and $\vec{x}$ are $u\times1$ and $v\times1$ matrices
respectively. Now $y_{i}\in z^{-1}U$ for $0\leq i<u$ implies that
$zy_{i}\in U$ for $0\leq i<u$. Let $\vec{x'}=z\vec{x}$ and $\vec{y'}=z\vec{y}$.
Then $\vec{y'}=A\vec{x'}$. So there exists $\vec{x'}$ with entries
from $S$ such that all entries of $A\vec{x'}$ are in $U$. Therefore
$p.q\in I(A)$. So $I(A)$ is a left ideal of $(0^{+},\cdot)$.

(b) Let $A$ be a $u\times v$ matrix, where $u,v\in\mathbb{N}$.
By previous lemma , $I(A)$ is a left ideal. We now show that $I(A)$
is a right ideal of $(0^{+},\cdot)$. Let $p\in\beta S$ and $q\in I(A)$.
Now let $U\in q\cdot p$. Then $\{x\in S:x^{-1}U\in p\}\in q$. So
there exists $\vec{x}$ with entries in $S$ such that $y_{i}\in\{x\in S:x^{-1}U\in p\}$
for $0\leq i<u$, where $\vec{y}=A\vec{x}$, $\vec{y}=\left(\begin{matrix}y_{0}\\
\vdots\\
y_{u-1}
\end{matrix}\right)$ and $\vec{x}=\left(\begin{matrix}x_{0}\\
\vdots\\
x_{v-1}
\end{matrix}\right)$; $\vec{y}$ and $\vec{x}$ are $u\times1$ and $v\times1$ matrices
respectively. Now for $0\leq i<u,y_{i}\in\{x\in S:x^{-1}U\in p\}$.
Hence $y_{i}^{-1}U\in p$ for $0\leq i<u$. This implies $\bigcap_{i=0}^{u-1}y_{i}^{-1}U\in p$.
So $\bigcap_{i=0}^{u-1}y_{i}^{-1}U\neq\emptyset$. Let $z\in\bigcap_{i=0}^{u-1}y_{i}^{-1}U$.
Therefore $z\in y_{i}^{-1}U$ for all $i\in\{0,1,2,.....,u-1\}$.
Hence $y_{i}z\in U$ for $0\leq i<u$. Let $\vec{x'}=\vec{x}z$ and
$\vec{y'}=\vec{y}z$. Then $\vec{y'}=A\vec{x'}$. So there exists
$\vec{x'}$ with entries from $S$ such that all entries of $A\vec{x'}$
are in $U$. Thus $q\cdot p\in I(A)$. Therefore $I(A)$ is also a
right ideal of $(\beta\mathbb{N},\cdot)$. Hence $I(A)$ is a two-sided
ideal of $(0^{+},\cdot)$.
\end{proof}
\noindent \begin{verbatim}


\end{verbatim}
\begin{proof}[Alternative proof of Theorem 2.2.7]
\noindent  Let $r\in\mathbb{N}$ be given and $\epsilon>0$. Let
$\mathbb{Q}=\bigcup_{i=1}^{r}E_{i}$. Suppose that $A$ be a $u\times v$
matrix where $u,v\in\mathbb{N}$. Also let $A=\left(\begin{matrix}M & 0\\
0 & N
\end{matrix}\right)$. Now by previous Lemma \ref{lem:I(A)}, $I(M)$ is a two-sided ideal
of $(0^{+},\cdot)$. So $K(0^{+},\cdot)\subseteq I(A)$. Also by same
Lemma \ref{lem:I(A)}, $I(N)$ is a left ideal of $(0^{+},\cdot)$.
Therefore $K(0^{+},\cdot)\cap I(N)\neq\emptyset$. Hence $I(M)\cap I(N)\neq\emptyset$.
Now choose $p\in I(M)\cap I(N)$. Since $\mathbb{Q}=\bigcup_{i=1}^{r}E_{i}$,
there exists $k\in\{1,2,.....,r\}$ such that $E_{k}\in p$. Thus
by definition of $I(M)$ and $I(N)$, there exist $\vec{x}\in S^{v}$
and $\vec{y}\in S^{\omega}$ such that $M\vec{x}\in E_{k}^{u}$ and
$N\vec{y}\in E_{k}^{\omega}$. Take$\vec{z}=\left(\begin{matrix}\vec{x}\\
\vec{y}
\end{matrix}\right)$. Then $A\vec{z}=\left(\begin{matrix}M\vec{x}\\
N\vec{y}
\end{matrix}\right)$. So $A\vec{z}\in E_{k}^{\omega}$. Therefore $A=\left(\begin{matrix}M & 0\\
0 & N
\end{matrix}\right)$ is image partition regular near 0.
\end{proof}
\bibliographystyle{amsplain}
\addcontentsline{toc}{section}{\refname}\bibliography{Paper_3}

\end{document}